\documentclass[12pt]{article}

\usepackage[left=1in,right=1in,top=1in,bottom=1in]{geometry}
\usepackage{dsfont}
\usepackage{amssymb}
\usepackage{amsmath}
\usepackage{amsthm}
\usepackage{amsfonts}
\usepackage{graphicx}
\usepackage{epstopdf}
\usepackage{algorithm}
\usepackage{algorithmic}
\usepackage{hyperref}
\DeclareMathOperator*{\argmin}{\arg\!\min}
\DeclareMathOperator{\prox}{prox}
\newcommand{\R}{\mathbb R}

\newcommand{\innerprod}[2]{\left\langle #1,#2\right\rangle}
\newcommand{\B}[1]{\boldsymbol #1}
\newcommand{\abs}[1]{\left\lvert #1 \right \rvert}
\newcommand{\g}{\gamma}
\newcounter{thmCounter}[subsection]

\hypersetup{ colorlinks=true, urlcolor=blue, linkcolor=red}

\newcommand{\dref}[1]{definition~\hyperref[def:#1]{\color{red}#1}}
\newcommand{\tref}[1]{theorem~\hyperref[thm:#1]{\color{red}#1}}
\newcommand{\Dref}[1]{definition~\hyperref[def:#1]{\color{red}#1}}
\newcommand{\Tref}[1]{theorem~\hyperref[thm:#1]{\color{red}#1}}
\newcommand{\lref}[1]{lemma~\hyperref[lem:#1]{\color{red}#1}}
\newcommand{\Lref}[1]{Lemma~\hyperref[lem:#1]{\color{red}#1}}
\newcommand{\pref}[1]{proposition~\hyperref[prop:#1]{\color{red}#1}}
\newcommand{\Pref}[1]{Proposition~\hyperref[prop:#1]{\color{red}#1}}
\newcommand{\eref}[1]{example~\hyperref[ex:#1]{\color{red}#1}}
\newcommand{\Eref}[1]{Example~\hyperref[ex:#1]{\color{red}#1}}

\newtheorem{theorem}{Theorem}[section]
\newtheorem{remark}{Remark}[section]
\newtheorem{definition}{Definition}[section]

\newcommand\blfootnote[1]{%
  \begingroup
  \renewcommand\thefootnote{}\footnote{#1}%
  \addtocounter{footnote}{-1}%
  \endgroup
}

\begin{document}

\begin{center}
    {\Large A Scalable Method for Optimal Path Planning on Manifolds via a Hopf-Lax Type Formula}\\

    \vspace{0.5cm}
    
    {\large Edward Huynh\footnote{Oden Institute, University of Texas at Austin, Austin, TX (edhuynh@utexas.edu)} and Christian Parkinson\footnote{Department of Mathematics \& Department of Computational Mathematics, Science, and Engineering, Michigan State University, East Lansing, MI (chparkin@msu.edu)}\blfootnote{The first author is funded under the National Defense Science and Engineering Graduate (NDSEG) Fellowship administered through the Department of the Air Force (AFRL/Space Force). The second author was partially supported by NSF DMS-1937229 through the Data Driven Discovery Research Training Group at the University of Arizona.}}
\end{center}

\begin{abstract}
We consider the problem of optimal path planning on a manifold which is the image of a smooth function. Optimal path-planning is of crucial importance for motion planning, image processing, and statistical data analysis. In this work, we consider a particle lying on the graph of a smooth function that seeks to navigate from some initial point to another point on the manifold in minimal time. We model the problem using optimal control theory, the dynamic programming principle, and a Hamilton-Jacobi-Bellman equation. We then design a novel primal dual hybrid gradient inspired algorithm that resolves the solution efficiently based on a generalized Hopf-Lax type formula. We present examples which demonstrate the effectiveness and efficiency of the algorithm. Finally, we demonstrate that, because the algorithm does not rely on grid-based numerical methods for partial differential equations, it scales well for high-dimensional problems.
\end{abstract}

\section{Introduction}
The increasing relevance of autonomous vehicles has led to developments in technology to maneuver them to their respective destinations. Simple models for modeling vehicle motion (e.g. Dubins \cite{Dubins1957} and Reeds-Shepps cars \cite{reeds}) have been developed and studied to determine optimal paths. This problem is an instance of path planning problems. Path planning in general has applications in motion planning \cite{Sanchez2021}, computer vision \cite{peyre2010} and grid refinement \cite{Qian2006,Ren2000}. In particular, path planning on manifolds is important in optics \cite{Hu2015} and image segmentation and vessel tracking \cite{duits2018optimal}.

Path planning using a PDE-based approach has been studied in various contexts including vehicular motion \cite{Vlad1,Vlad2,parkinson2024efficient, parkinson2021timeoptimal,Takei2013OptimalTO,takei2010}, isotropic motion \cite{ParkPolage}, human walking paths \cite{SteepTerrain2,SteepTerrain1}, models for environmental crime \cite{Arnold1, Cartee, Bohan} and robotics \cite{Borquez,ChoiBansal} to list a few. PDE-based optimal path planning arises naturally from the dynamic programming approach to optimal control theory, wherein one defined a value function which satisfies a Hamilton-Jacobi-Bellman (HJB) equation. The solution to this equation determines closed-loop feedback controls which are globally optimal. While recent work in path planning has been exploring the use of deep learning and data-driven approaches \cite{KULATHUNGA2022152, Singh2023}, the advantage of working in this paradigm is that there are no black box elements and classical analysis of PDE can provide theoretical guarantees regarding optimality and robustness.

Solving the path planning problem in this approach means being able to solve the corresponding HJB equation. Among the approaches to numerically solve Hamilton-Jacobi equations, important methods include the level-set method \cite{levelSet}, fast-sweeping \cite{kaoosher2005,luo2016convergence,parkinson2021rotating,tsai2003fast}, and fast-marching schemes \cite{Tsitsiklis1995,Sethian2000,sethian2003ordered}. These methods rely on discretizing the spatial domain and implementing finite differences to estimate derivatives. However, the computational complexity of these algorithms scales exponentially with spatial dimension, making them infeasible in high spatial dimensions. In recent years, due to applications involving high-dimensional control and differential games, variational methods which can compute the solution to a PDE at individual points have attracted substantial interest. The authors of \cite{darbon2016} present a method for solving eikonal equations using the Hopf-Lax formula. Later, similar methods were developed based on conjectural extensions of the Hopf-Lax formula to state dependent Hamiltonians \cite{CHOW2019376, lin2018splitting}.

In this paper, we propose a new formulation for computing optimal trajectories on manifolds via Hamilton-Jacobi-Bellman equations. Optimal trajectories can be resolved efficiently by discretizing time only, and formulating a saddle-point problem for optimal state and co-state trajectories similar to that presented in \cite{lin2018splitting}. We design an algorithm for solving the saddle-point problem inspired by the primal dual hybrid gradient (PDHG) algorithm of Chambolle and Pock \cite{Chambolle2011AFP}.  We specifically address path planning on manifolds which are graphs of $C^1$-functions, though with slight modifications it could likely be generalized to other manifolds. Our algorithm is a novel modification based on preconditioning of PDHG \cite{liu2021acceleration}. This approach is compatible with the time-dependent formulation for stationary/moving obstacles given in \cite{parkinson2024efficient}, and could account for a manifold which is evolving in time.

This manuscript is organized as follows: section 2 describes the mathematical formulation of the problem; section 3 introduces the variational formulation of the problem and the numerical methods which we design to solve it; section 4 presents some examples, results and discussion; section 5 includes some brief concluding remarks.   

\section{Particle moving on a manifold}

We consider the problem of minimal-time path planning on an $n$-dimensional manifold $\mathcal M$ embedded in $\R^{n+1}$ as the graph of a smooth function $M:\R^n \to \R$. We denote points on the manifold by $(x,z)$ where $x\in \R^n$ is the variable parameterizing the manifold and $z = M(x) \in \R$.  Suppose at time $t=0$, a particle is at position $(x_0,z_0)\in \mathcal{M}$ and an external user seeks to steer the particle to position $(x_f,z_f) \in \mathcal M$ in minimal time. We consider a particle exhibiting isotropic motion in $\R^n$ so that the dynamics are controlled by a unit vector $\B{a}(\cdot) \in \mathbb{S}^{n-1}$. This is pictured in Figure \ref{fig:0}. We give a brief derivation of our model here.

\begin{figure}[b!]
    \centering
    \includegraphics[width=0.6\textwidth,trim = 50 20 40 20,clip]{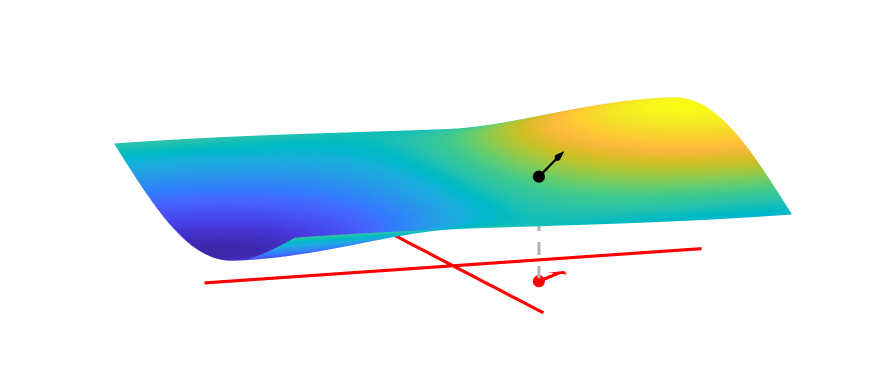}
    \caption{A particle traveling on a manifold, and its projection down to $\R^{n}$ (represented by the plane). Here $\B x(t)$ is the red dot and $(\B x(t),z(t))$ is the black dot. The red arrow is the control variable $\B a(t)$, which is a direction vector in $\R^n$. This choice of $\B a(t)$ induces a motion (black arrow) along the manifold.}
    \label{fig:0}
\end{figure}

Fix a time horizon $T > 0$, and suppose the position of the particle projected down to $\R^n$ is given by $\B x:[0,T] \to \R^n$, so that the particle's position in $\R^{n+1}$ is $(\B x(t), \B z(t)) \in \mathcal M$ where $\B z(t) := M(\B x(t)) \in \R$.  We fix an orientation on the manifold by defining the unit normal  $$\B \nu = \frac{(\nabla M,-1)}{\sqrt{1+ \abs{\nabla M}^2}}.$$ A trajectory along the manifold whose $\B x(t)$-coordinates are moving in the direction of $\B a(t)$ then has unit tangent vector $$\B \tau(t) = \frac{(\B a(t),\innerprod{\nabla M(\B x(t))}{\B a(t)})}{\sqrt{1 + \innerprod{\nabla M(\B x(t))}{\B a(t)}^2}}.$$ For brevity, we put $$q(\B x,\B a) = \frac{1}{\sqrt{1 + \innerprod{\nabla M(\B x)}{\B a}^2}}.$$ We assume that there is a speed function $v:\R^n \times [0,T] \to [0,\infty)$ such that a particle which is at $x \in \R^n$ at time $t \in [0,T]$ can move with local speed $v(x,t)$. The equation of motion in $\R^n$ is then \begin{equation} \label{eq:dynamics1}
\begin{split}
\dot{\B x}(t) = v(\B x(t),t) q(\B x(t),\B a(t)) \B a(t) \mathds{1}_{x \neq x_f}(\B x(t)).
\end{split}
\end{equation} Here the indicator function $\mathds{1}_{x\neq x_f}(x)$ is somewhat artificial. Including this in the dynamics is a manner of forcing the particle to stop when the desired final point is reached. As we will see shortly, it is also convenient because it allows us to formally derive the Hamilton-Jacobi-Bellman equation for our optimal control problem without the use of boundary conditions, which will be helpful for our numerics. We discuss this further in Section \ref{sec:NumericalApproach}. The equation of motion for the $z$-coordinate is then $$\dot{\B z}(t) = \innerprod{\nabla M(\B x(t))}{ \dot{\B x}(t)}.$$ However, one need not actually use this equation, since $\B z(t)$ can simply be defined by $\B z(t) = M(\B x(t))$. 

As stated above, the goal of the controller is to steer the particle to $(x_f,z_f)$ in minimal time. To model this mathematically, we consider a controller who seeks to minimize the cost functional \begin{equation} \label{eq:costFunc}\mathcal C[\B x(\cdot),\B a(\cdot)] = \iota_{x_f}(\B x(T)) + \int^T_0 \mathds{1}_{x\neq x_f}(\B x(t))dt\end{equation} where \begin{equation} \label{eq:convInd} \iota_{x_f}(x) = \begin{cases} 0, & x=x_f, \\ +\infty, & \text{otherwise}, \end{cases} \end{equation} is the convex indicator function of the final point $x_f$. Intuitively, using the convex indicator as the exit cost in \eqref{eq:costFunc} ensures that no path which fails to reach $x_f$ will be optimal. Using the indicator function $\mathds{1}_{x\neq x_f}(x)$ as the marginal running cost ensures that running cost is merely time spent traveling, and will stop counting when the final point is reached. Thus the optimal cost value is simply the minimal travel time.

To use the classical dynamical programming approach of Richard Bellman \cite{Bellman1952}, we fix a point $(x,t) \in \R^n \times [0,T)$, define the value function $u(x,t)$ to be the optimal \textit{remaining} cost incurred by a trajectory $\B x(\cdot)$ which is at position $x$ at time $t$. That is, $u:\R^n \times[0,T] \to \R$ is defined \begin{equation} \label{eq:valueFunc}u(x,t) = \inf_{\substack{\B a(s) \in \mathbb S^{n-1} \\ t\le s \le T}} \mathcal C_{x,t}[\B x(\cdot),\B a(\cdot)]\end{equation} where $\mathcal C_{x,t}$ is the same cost functional as before but restricted to the time interval $[t,T]$ and to trajectories which are at position $x$ at time $t$. 

Formally, applying the dynamic programming principle, one sees that the value function is the solution of the Hamilton-Jacobi-Bellman (HJB) equation \begin{equation}
    \label{eq:HJB1} \begin{split}
    &u_t + \mathds{1}_{x\neq x_f}(x) \inf_{a \in \mathbb S^{n-1}} \Big\{ v(x,t)q(x,a) \innerprod{a}{\nabla u}+1 \Big\} = 0, \,\,\,\,\,\, (x,t) \in \R^n \times [0,T)\\
    &u(x,T) = \iota_{x_f}(x), \,\,\,\,\,\,\,\,\, x \in \R^n.
    \end{split}
\end{equation} The derivation of the HJB equation---as well as a discussion of viscosity solutions for Hamilton-Jacobi equations \cite{Visc1}---is included in several references with varying levels of rigor \cite{Bertsekas,Bryson,FlemingRishel,Liberzon}. Ordinarily, this equation would be paired with a boundary condition that $u(x_f,t) = 0$ for all time $0\le t\le T$. Formally, since $u(x_f,t) = 0$ and $u_t = 0$  at $x_f$ due to the indicator function, our solution will satisfy this boundary condition. We use the indicator function in \eqref{eq:HJB1}, because our numerical methods are more amenable to problems with no boundary conditions. We employ a smooth approximation to this indicator function below. 

In much of the ensuing discussion, we use the variable $p \in \R^n$ as a proxy for the gradient $\nabla u$. We define the Hamiltonian \begin{equation}
    \label{eq:Hamiltonian1} 
    H(x,p,t) = \inf_{a \in \mathbb S^{n-1}} \Big\{ v(x,t)q(x,a) \innerprod{a}{p}+1 \Big\}
\end{equation}  
The HJB equation is most easily utilized when the Hamiltonian can be evaluated explicitly (that is, when the infimum in \eqref{eq:Hamiltonian1} can be resolved explicitly) and this is possible in our case.

\begin{theorem} \label{thm:HJBEqn}
 Equation \eqref{eq:HJB1} is equivalent to \begin{equation} \label{eq:HJB1exp}
\begin{split}
&u_t - \mathds{1}_{x\neq x_f}(x) H(x,\nabla u,t)= 0, \\
&u(x,T) = \iota_{x_f}(x),
\end{split} 
\end{equation} where the Hamiltonian is given by 
\begin{equation} \label{eq:Hamiltonian}
H(x,p,t) = v(x,t)\sqrt{p^T\left (I - \frac{\nabla M(x) \nabla M(x)^T}{1 + \abs{\nabla M(x)}^2}\right)p}- 1.
\end{equation}
Moreover, the matrix $I - \frac{\nabla M(x) \nabla M(x)^T}{1 + \abs{\nabla M(x)}^2}$ is symmetric and positive definite (uniformly in $x$ as long as $\nabla M(x)$ is bounded), and as such $H(x,p,t)$ is convex in $p$.
\end{theorem}

\begin{proof} For the sake of the minimization in \eqref{eq:Hamiltonian1}, we can fix $x,p$ (and thus $\nabla M$), at which point the derivation of the Hamiltonian boils down to minimizing $$f(a) = \frac{\innerprod p a}{\sqrt{1+\innerprod{\nabla M}a^2}} \,\,\,\,\, \text{ such that } \,\,\,\,\, \abs a^2 = 1.$$ This is a relatively straightforward, if long and tedious, task so we relegate it to Appendix \ref{nHJBDeriv} and accept here that \eqref{eq:Hamiltonian} is correct.

 Given this, for the remainder of the manuscript, we define \begin{equation} \label{eq:Amat}A(x) = I - \frac{\nabla M(x) \nabla M(x)^T}{1 + \abs{\nabla M(x)}^2}.\end{equation} From this formula, it is clear that $A(x)$ is symmetric. Because it is a rank 1 perturbation of the identity, we can explicitly see that the eigenvalues are given by $\lambda_1 = 1$ which has as eigenvectors the set of all vectors orthogonal to $\nabla M(x)$ and thus has multiplicity $n-1$ in dimension $n$, and $\lambda_2 = \frac{1}{1+\abs{\nabla M(x)}^2}$ with eigenvector $\nabla M(x)$. Thus, $A(x)$ is positive definite (uniformly in $x$ when $\nabla M(x)$ remains bounded) and $\sqrt{p^T A(x) p}$ defines a norm on vectors $p$. It follows that the Hamiltonian \eqref{eq:Hamiltonian} is convex in $p$.
\end{proof}

This equation admits a unique viscosity solution that is related to the distance function on a Riemannian manifold \cite{mantegazza2002hamilton}. We note that in the case that the surface is flat so that $\nabla M \equiv 0$, we have \begin{equation}
\label{eq:EikonalHamiltonian}
H(x,p,t) = v(x) \abs{p} - 1
\end{equation} so our equation reduces to the standard Eikonal equation. With our definition for the matrix $A(x)$ in \eqref{eq:Amat}, we will henceforth write our Hamiltonian \eqref{eq:Hamiltonian} as \begin{equation} \label{eq:HamVec} H(x,p,t) =v(x,t) \sqrt{p^TA(x)p}- 1 \end{equation}

\section{Numerical Approach}\label{sec:NumericalApproach}
In this section, we will discuss the numerical methods we use to efficiently solve this problem. To do this, we first mention in Section \ref{sec:HopfLax} the representation formula that is used to obtain a numerical solution. Section \ref{sec:PDHG} will explain the splitting algorithm that is used to approximately solve the requisite saddle-point problem. and Section \ref{sec:ApproxSpace} will provide a theoretical justification for the choice of our initialization.

\subsection{A Generalized Hopf-Lax Formula} \label{sec:HopfLax} We consider the value function for a generic optimal control problem given by
\begin{gather}\label{OptimalControl}
    u(x,t) = \inf_{\B x(\cdot), \B a(\cdot)} \left\{ g(\B x(T)) + \int_t^T r(\B x,\B a,s)ds\,  : \, \B x(t) = x, \dot{\B x} = f(\B x, \B a, s) \text{ on } (t,T] \right\}.
\end{gather} Here $g$ and $r$ represent the terminal cost and running cost, respectively. Under mild conditions on the data $u(x,t)$ is the viscosity solution of the Hamilton-Jacobi-Bellman equation $u_t + H(x,\nabla u,t) = 0$ subject to the terminal condition $u(x,T) = g(x)$, where $H$ is given by $H(x, p,t) = \inf_{\alpha} \{\langle p, f(x,a,t) \rangle + r(x,a,t)\}$ \cite{EvansControl,FlemingRishel}. As in \cite{lin2018splitting,parkinson2024efficient}, we make the time-reversing substitution $t \mapsto T-t$, to arrive at a forward-in-time HJB equation \begin{equation} \label{eq:generalHJ}\begin{split} 
    u_t + \hat H(x,\nabla u,t) &= 0,\\
    u(x,0) &= g(x).
    \end{split}
\end{equation} Here $\hat H(x,p,t) = -H(x,-p,T-t)$ and the solution of \eqref{eq:generalHJ} is the time-reversed version of the value function defined in \eqref{OptimalControl}. In an abuse of notation, we drop the hat and also do not rename $u$. 

In the case where the Hamiltonian is time- and state-independent, the classical Hopf-Lax formula provides a representation of the value function at individual points $(x,t)$ in terms of a minimization problem \cite[Chap. 3]{EvansPDE}. The authors of \cite{CHOW2019376} present a conjectural generalized Hopf-Lax formula for time- and state-dependent Hamiltonians, which essentially boils down to optimizing \eqref{OptimalControl} over all bi-characteristic state curves $\B x(s)$ and costate curves $\B p(s) = \nabla u(x(s),s)$. While the specific conjecture of \cite{CHOW2019376} applies to $C^2$ Hamiltonians which are convex in $p$ and convex, coercive initial data, they specifically address the case of so-called \emph{level set equations} \cite{levelSet}, wherein the Hamiltonian is positively homogeneous of degree 1 in $p$, and prove that for small time this formula can represent the viscosity solution. Likewise, the authors of \cite{lin2018splitting} propose a discrete approximation of \eqref{OptimalControl}, which again, is merely conjectured to approximate the viscosity solution of \eqref{eq:generalHJ}, though they provide solid empirical evidence, which is corroborated by \cite{parkinson2024efficient}, that their discrete formulation can indeed represent the viscosity solution of eikonal-type equations. 

For the sake of completeness, we derive a discrete saddle-point formulation of \eqref{OptimalControl} very similar to those presented in \cite{lin2018splitting,parkinson2024efficient}. Fixing $J \in \mathbb N$, we use a uniform discretization of  $(t,T]$ with step size $\Delta t = (T-t)/J$ and let $x_j$ denote the discrete approximation of $x(t_j)$, and $a_j$ denote the control action at $t_j$ for $j = 0,1,\ldots, J$ where $t_j = t+j\Delta t$. We then discretize the dynamics with backward Euler: $\dot{x} \approx \frac{x(t) - x(t-\Delta t)}{\Delta t}$. Letting $f_j = f(x_j, a_j,t_j), r_j = r(x_j,a_j,t_j)$ and introducing Lagrange multipliers $p_j$ on the constraints, we obtain
\begin{align*}
    u(x,t) &= \inf_{\B x, \B a}\left\{g(\B x(T)) + \int_t^{T} r(\B x,\B a, s) ds:\ \B x(t) = x,\ \dot{\B x} = f(\B x,\B a,s),\ s\in (t,T]\right\} \\
    &\approx \inf_{x_j,a_j}\sup_{p_j}\left\{g(x_J) + \sum_{j=1}^{J} \Delta t r_j +  \sum^J_{j=1}\langle p_j, \Delta t f_j +x_{j-1} - x_j \rangle\right\}\\
    &= \inf_{x_j}\inf_{a_j}\sup_{p_j}\left\{g(x_J) + \sum_{j=1}^J \langle p_j, x_{j-1} - x_{j}\rangle  + \Delta t \sum_{j=1}^{J}\Big( \langle p_j, f_j\rangle + r_j\Big)\right\}.
\end{align*}
Operating formally, we then interchange the infimum in $a_j$ with the supremum in $p_j$ to obtain
\begin{align*}
    u(x,t) &\approx \inf_{x_j}\sup_{p_j}\left\{g(x_J) + \sum_{j=1}^J \langle p_j, x_{j-1} - x_{j}\rangle  + \Delta t \sum_{j=1}^{J}\inf_{a_j} \Big( \langle p_j, f_j\rangle+r_j\Big)\right\}\\
    &= \inf_{x_j}\sup_{p_j}\left\{g(x_J) + \sum_{j=1}^J \langle p_j, x_{j-1} - x_{j}\rangle  + \Delta t \sum_{j=1}^{J} H(x_j,a_j,t_j)\right\}
\end{align*}
Finally, we apply the backwards in time substitution $t \mapsto T-t$ (and let $x_j$ and $p_j$ now represent time-reversed trajectories) to arrive at 
\begin{equation} \label{eq:generalizedHopfLax}
    u(x,t) \approx \inf_{x_j}\sup_{p_j}\left\{g(x_0) + \sum_{j=1}^J \langle p_j, x_{j} - x_{j-1}\rangle  - \Delta t \sum_{j=1}^{J} H(t_j, x_j, p_j)\right\},
\end{equation}
which is the discretized version of \eqref{OptimalControl}. We point out the formal similarity with the classical Hopf-Lax formula \cite[Chap. 3]{EvansPDE} in the case that $H(x,p,t) = H(p)$:
$$u(x,t) = \inf_{x_0 \in \R^n} \left\{ g(x_0) + t H^*\left(\frac{x-x_0}{t}\right)\right\} =  \inf_{x_0 \in \R^n} \sup_{p \in \R^n} \{g(x_0) + \langle p, x-x_0 \rangle -  t H(p)\}.$$ Above $H^*$ is the convex conjugate of $H$ (which we expand in the second equality). This formula likewise provides the solution of the Hamilton-Jacobi equation at a single point in the form of a saddle-point problem. The key difference is that when the Hamiltonian is time- and space-independent, the co-state characteristics are straight lines, which is why the minimization and maximization take place over single values (in essence the starting point of the state trajectory and ending point of the co-state trajectory), as opposed to taking place over the entire discretized trajectory as in \eqref{eq:generalizedHopfLax}.


Swapping the order of optimization to arrive at \eqref{eq:generalizedHopfLax} is licit under certain convexity conditions (e.g., the hypotheses of von Neumann's Minimax Theorem \cite{v1928theorie}). In general, we don't necessarily expect our data to satisfy these conditions. but conjecture, as in \cite{CHOW2019376,lin2018splitting,parkinson2024efficient,ParkPolage} that \eqref{eq:generalizedHopfLax} can still represent the viscosity solution of \eqref{eq:HJB1exp}. 

Stacking all state vectors into a single vector $\tilde x = (x_0^T,x_1^T,\ldots, x_N^T)^T$ and likewise for costate vectors, \eqref{eq:generalizedHopfLax} can be written like
\begin{equation} \label{eq:SPP}
    u (x,t) \approx \min_{\tilde x \in \R^{(J+1)n}} \max_{\tilde p \in \R^{(J+1)n}} \bigg\{\tilde{G}(\tilde{x}) + \langle \tilde{p}, D\tilde{x} \rangle + \tilde{H}(\tilde{x}, \tilde{p}, \tilde{t}) \bigg\},
\end{equation}
where $n$ is the dimension of the state space, $D$ is the matrix which accomplishes the backwards difference $x_j - x_{j-1}$ in \eqref{eq:generalizedHopfLax}, $\tilde{H}(\tilde x,\tilde p,\tilde t)$ is the discrete approximation to the integral along the path given in \eqref{eq:generalizedHopfLax}, and $G(\tilde x) = g(x_0)$. We write it in this form to demonstrate that this expresses $u(x,t)$ as the solution to a saddle-point problem very similar to the type of general saddle-point problem considered by Chambolle and Pock \cite{Chambolle2011AFP} in their seminal paper where they introduce the Primal Dual Hybrid Gradient (PDHG) algorithm. Inspired by this, we design a PDHG type algorithm to approximate \eqref{eq:SPP}. 

\subsection{A PDHG Algorithm for \eqref{eq:SPP}} \label{sec:PDHG}
The original PDHG algorithm \cite{Chambolle2011AFP} allows for efficient solution to certain saddle-point problems. The algorithm was specifically designed for problems of the form \eqref{eq:SPP} where $\tilde H(\tilde x, \tilde p,\tilde t)$ is state-independent, but has been successfully applied by \cite{lin2018splitting,parkinson2024efficient,ParkPolage} to problems with state-dependent Hamiltonian. For our particular case, we make one final simplification before designing our PDHG-type algorithm. 

For the HJB equation \eqref{eq:HJB1exp}, \eqref{eq:generalizedHopfLax} can be written \begin{equation} \label{eq:p} \begin{split}u(x,t) &\approx \inf_{\{x_j\}}\sup_{\{p_j\}} \bigg\{g(x_0) + \sum^J_{j=1} \innerprod{p_j}{x_{j}-x_{j-1}} \\ &\hspace{2.5cm}- \Delta t \sum^J_{j=1} \mathds 1_{x\neq x_f}(x_j)(v(x_j,t_j) \innerprod{p_j}{A(x_j)p_j} -1) \bigg\} \end{split}\end{equation} For reasons that will be elucidated shortly, it will be convenient to make a change of variables in the costate variable. Specifically, because $A(x)$, as defined in \eqref{eq:Amat}, is symmetric and positive definite (uniformly in $x$ so long as $\nabla M(x)$ is bounded), we have a Cholesky factorization $A(x) = L(x)L(x)^T$ where $L(x)$ is lower triangular with positive diagonal elements. Defining $w_j = L(x_j)^Tp_j$, we see \begin{equation}\label{eq:w}\begin{split} u(x,t) &\approx \inf_{\{x_j\}}\sup_{\{w_j\}} \bigg\{g(x_0) + \sum^J_{j=1} \innerprod{w_j}{L(x_j)^{-1}(x_{j}-x_{j-1})} \\ &\hspace{2.5cm}- \Delta t \sum^J_{j=1} \mathds 1_{x\neq x_f}(x_j)(v(x_j,t_j) \abs{w_j} -1) \bigg\}. \end{split}\end{equation} 

Roughly speaking, the PDHG algorithm operates by alternating between maximization in $p$ (or $w$) and minimization in $x$. A convenient facet of our saddle point problems is that if $\{p_j\}$ is fixed, the minimization with respect to $\{x_j\}$ in \eqref{eq:p} can be entirely de-coupled and solved to the individual vectors $x_j$. Likewise, if $\{x_j\}$ is fixed, the maximization over $\{w_j\}$ in \eqref{eq:w} can be performed separately for each $w_j$. Thus, we numerically solve our saddle point problem using Algorithm \ref{Algorithm1}. 

Because it has bearing here and will be useful later, we recall the definition of the proximal operator.  \begin{definition}
    Suppose $f:\R^n \to \R$ is a proper, convex, and lower semi-continuous function. Then for each $y \in \R^n$, we define the proximal operator of $f$ by
    $$\prox_{f}(y) = \argmin_{x\in \R^n} \bigg\{f(x) + \frac{1}{2}|x - y|^2 \bigg\}$$
\end{definition} Using this notation, the Algorithm \ref{Algorithm1} essentially boils down to iterating \begin{equation} \label{eq:proxs}
    \begin{split}
        w^{k+1}_j &= \prox_{\sigma \Delta t H(x^k_j,(L(x_j^k)^T)^{-1}(\,\cdot\,),t_j)}(\beta_j^k), \\
        p^{k+1}_j &= (L(x_j^k)^T)^{-1}w^{k+1}_j, \\
        x^{k+1}_j &= \prox_{-\tau\Delta t H(\,\cdot\,,p_j^{k+1},t_j)}(\nu_j^k),
    \end{split}
\end{equation} where $\beta^k_j, \nu_j^k$ are as defined in Algorithm \ref{Algorithm1}. One of the key issues of implementing this algorithm directly is that these proximal operators do not in general admit solutions which can be explicitly resolved as simple formulas of the data. The reason that we make the change of variables is that the minimization which defines $w_j^{k+1}$ can now be explicitly resolved in the case of our particular Hamiltonian. We see  \begin{equation} \label{eq:resolvew} \begin{split} 
w^{k+1}_j &= \max\left(0,1 - \frac{\sigma \Delta t \mathds 1_{x\neq x_f}(x^k_j) v(x^k_j,t_j)}{\abs{\beta^k_j}}\right)\beta^k_j \\ 
&\text{where } \beta^k_j = w_j^k +\sigma L(x_j^k)^{-1}(z_j^k - z_{j-1}^k).
\end{split} \end{equation} This formula comes from the proximal operator of the Euclidean norm; derivations can be found in \cite{beck,parkinson2024efficient}.

\begin{algorithm}[t!]
\caption{Splitting Method for \eqref{eq:generalizedHopfLax}}
\hspace*{\algorithmicindent} Given a point $(x,t) \in \mathbb{R}^d\times (0,T)$, a Hamiltonian $H$, an initial data function $g$, a time-discretization count $J$, a max iteration count $K$, an error tolerance $TOL$, and PDHG parameters $\sigma,\ \tau >0$ and $\kappa \in [0,1]$, we approximately solve \eqref{eq:generalizedHopfLax} as follows:\\

Set $x_J^1 = x$, $w_0^1 = 0$, and $\Delta t = t/N$, $t_j = j\Delta t$. Initialize $\{x^0_j\}_{j=0}^{J-1}, \{p_j^0\}^J_{j=1}$ randomly, and set $\{z_j^0\} = \{x_j^0\}, \{w_j^0\} = \{p_j^0\}$
\begin{algorithmic}
\FOR{$k = 0:K$}
\STATE $w_0^{k+1} = 0, \,\, p_0^{k+1}=0$ 
\FOR{$j = 1:J$}
\STATE $\beta^k_j = w_j^k +\sigma L(x_j^k)^{-1}(z_j^k - z_{j-1}^k)$
\STATE $w_j^{k+1} = \argmin_{\tilde{w}} \{\sigma \Delta t  H(x_j^k, (L(x_j^k)^T)^{-1}\tilde{w},t_j) + \frac12|\tilde{w} - \beta^k_j|^2\}$ 
\STATE $p_j^{k+1} = (L(x_j^{k})^T)^{-1}w_j^{k+1}$
\ENDFOR
\STATE $x_0^{k+1} = \argmin_{\tilde{x}}\{ \tau g(\tilde{x}) + \frac12|\tilde{x} - (x_0^k + \tau p_1^{k+1})|^2$\quad (note: $p_0^{k+1} = 0$)
\FOR{$j = 1:J-1$}
\STATE $\nu_j^k = x_j^k - \tau(p_j^{k+1} - p_{j+1}^{k+1})$
\STATE $x_j^{k+1} = \argmin_{\tilde{x}}\{-\tau\Delta t H(\tilde{x}, p_j^{k+1},t_j) + \frac12|\tilde{x} - \nu^k_j|^2\}$
\ENDFOR
\STATE $x_J^{k+1} = x$
\FOR{$j=0:J$}
\STATE $z_j^{k+1} = x_j^{k+1} + \kappa(x_j^{k+1} - x_j^k)$
\ENDFOR
\STATE $\text{change} = \max\{|x^{k+1} - x^k|, |p^{k+1} - p^k|\}$
\IF{$\text{change} < \text{TOL}$}
\STATE stop iteration
\ENDIF
\ENDFOR
\STATE $u = g(x_0) + \sum_{j=1}^N\langle p, x_j - x_{j-1}\rangle - \Delta t H(x_j, p_j,t_j)$
\RETURN $u$; the approximate value of the solution at the point $(x,t)$
\RETURN $\{x_j^{k+1}\}$; an approximation of the optimal trajectory
\RETURN $\{p_j^{k+1}\}$; an approximation of the optimal costate trajectory (from which optimal control values can be computed)
\end{algorithmic}\label{Algorithm1}
\end{algorithm}

While this explains why we make the change of variables, the choice of the Cholesky factorization perhaps deserves some justification. The calculations are similarly simplified using any factorization of the form $A(x) = L(x)L(x)^T$, and it may seem more natural at first blush to instead take $L(x)$ to be the positive square root of $A(x)$. However, for the sake of computation, this is actually inadvisable. We note that $A(x) \to I$ as $\nabla M(x) \to 0$. Because we need to invert the factor matrix $L(x)$ in Algorithm \ref{Algorithm1}, we would like $L(x)$ to respect this limit as well. That is, we want a factorization $A(x) = L(x)L(x)^T$ such that $L(x) \to I$ as $\nabla M(x) \to 0$. This is true of Cholesky factorization, as can be seen in the 2-dimensional case, where ignoring the prefactor $1/(1+\abs{\nabla M}^2)$, we have \begin{equation} \label{eq:chol}\begin{bmatrix}1 + M_y^2 & -M_xM_y \\ -M_xM_y & 1+M_x^2 \end{bmatrix} = \begin{bmatrix}
        \sqrt{1+M_y^2} &0\\
        \frac{-M_xM_y}{\sqrt{1+M_y^2}} & \sqrt{\frac{1+\abs{\nabla M}^2}{1+M_y^2}}
    \end{bmatrix} \begin{bmatrix}
        \sqrt{1+M_y^2} & \frac{-M_xM_y}{\sqrt{1+M_y^2}} \\
        0& \sqrt{\frac{1+\abs{\nabla M}^2}{1+M_y^2}}
    \end{bmatrix}.\end{equation} By contrast, any factorization based on the eigenvectors of $A(x)$ (one of which is $\nabla M(x)$) will degenerate as $\nabla M(x) \to 0$. This is seen, for example, in the 2-dimensional case if we use the positive square root: $$
    \begin{bmatrix}1 + M_y^2 & -M_xM_y \\ -M_xM_y & 1+M_x^2 \end{bmatrix} = \left(\frac{1}{\abs{\nabla M}^2}\begin{bmatrix}
        M_x^2 + \alpha M_y^2 & (\alpha-1)M_xM_y\\
        (\alpha-1)M_xM_y & \alpha M_x^2+M_y^2
    \end{bmatrix}\right)^2$$ where $\alpha = \sqrt{1+\abs{\nabla M}^2}.$ In this case, the factor matrix is not well-defined in the limit as $\nabla M(x) \to 0$, since the matrix of eigenvectors is normalized by the factor $1/\abs{\nabla M(x)}$.

To reiterate, using the factorization $A(x) = L(x)L(x)^T$ ensures that the minimization for $w^{k+1}_j$ (and thus the resolution of $p^{k+1}_j$) in Algorithm \ref{Algorithm1} is entirely explicit, in that we can write down the formula for the minimizer. Except in extremely simple cases (such as $v(x,t) = 1$ and $\nabla M(x) = 0$), we will not be able to resolve the minimization for $x_j^{k+1}$ explicitly. Accordingly, we need some approximation of the minimizer.

\subsection{Approximating the minimizer for $x^{k+1}_j$} \label{sec:ApproxSpace}

The problem of resolving the proximal operator which defines $x^{k+1}_j$ in \eqref{eq:proxs} is also encountered in \cite{lin2018splitting, parkinson2024efficient, ParkPolage}. Because it cannot be resolved explicitly, the authors suggest using a few steps of gradient descent. In our case, letting ${\mathcal H}_M(x)$ denote the Hessian of $M(x)$, and suppressing dependence on $x$, we see \begin{equation} \label{eq:gradA}\nabla_x\left(\sqrt{p^TAp}\right) = \frac{(p^T\nabla M)^2{\mathcal H}_M\nabla M - (1+\abs{\nabla M}^2)(p^T\nabla M){\mathcal H}_Mp}{\sqrt{p^TAp}\left(1+\abs{\nabla M}^2\right)^2}. \end{equation} Thus, looking at \eqref{eq:HJB1exp}, we see \begin{equation} \label{eq:gradDesc}\begin{split} 
\nabla_x\Big(\mathds 1_{x\neq x_f}(x)H(x,p,t)\Big) &= \nabla_x\left(\mathds 1_{x\neq x_f}(x)\left(v(x,t)\sqrt{p^TA(x)p}-1\right)\right)\\
&= \left(v(x,t)\sqrt{p^TA(x)p}-1\right) \nabla_x\left(\mathds 1_{x\neq x_f}(x)\right) \\ &\hspace{0.3in}+ \mathds 1_{x\neq x_f}(x)\sqrt{p^TA(x)p} \,\, \nabla_xv(x,t) \\ &\hspace{0.3in}+  \mathds 1_{x\neq x_f}(x) v(x,t) \nabla_x \left( \sqrt{p^TA(x)p}\right).
\end{split} \end{equation} We can then initialize $x^{k+1}_j$  however we like (we use $x^k_j$), and run a few iterations of \begin{equation} \label{eq:actualGradDesc}
x^{k+1}_j \leftarrow x^{k+1}_j - \eta \left(-\tau \Delta t \nabla_x\Big(\mathds 1_{x\neq x_f}(x^{k+1}_j)H(x^{k+1}_j,p^{k+1}_j,t_j)\Big) + (x^{k+1}_j - \nu^k_j)\right)
\end{equation} where $\eta >0$ is the gradient descent rate. A similar but much simpler formula could be used to approximate $x_0^{k+1}$ in Algorithm \ref{Algorithm1}, or depending on $g(x)$, there may be an explicit formula (in our case, since $g(x) = \iota_{x_f}(x)$ is the convex indicator of the $x_f$, we trivially have $x_0^{k+1} = x_f$). One last note regarding the indicator function $\mathds 1_{x\neq x_f}$: in order to use \eqref{eq:gradDesc}, one needs to approximate this indicator function with a smooth function such as \begin{equation} \label{eq:indApprox} \mathds 1 _{x\neq x_f}(x) \approx 1- e^{-B\abs{x-x_f}^2}\end{equation} for a large constant $B$. Alternatively, one could deal with this by separately checking if $x = x_f$ is the minimizer when updating $x^{k+1}_j$, in which case the indicator function can be replaced by $1$ in \eqref{eq:gradDesc}. We found that convergence of the algorithm was faster (both in terms of clock time and iteration count) when a smooth approximation is used. This approximation is also used in \eqref{eq:resolvew}.

Empirically, as na\"ive as this may seem, it works surprisingly well: both \cite{lin2018splitting,ParkPolage} report that a single step of gradient descent at each iteration is sufficient for their purposes. However, in our case, the computation of the gradient in \eqref{eq:gradDesc} is among the most computationally costly operations during each iteration in Algorithm \ref{Algorithm1}, so to increase efficiency (in terms of clock time), one may employ the even simpler approximation \begin{equation} \label{eq:proxApprox} x^{k+1}_j = \prox_{-\tau\Delta t H(\,\cdot\,,p_j^{k+1},t_j)}(\nu_j^k) \approx \nu^k_j.\end{equation} If $p^{k+1}_j$ has already converged so that it is held constant with respect to the iteration count $k$, then this amounts to resolving $x^{k+1}_j$ via fixed point iteration on the proximal operator which is known to converge via the Kranosel'ski\u\i-Mann theorem \cite{Dong,Krano,Mann}. So there is reason to believe that after many iterations, this will be a good approximation. That being said, we would still like to quantify the error. In order to do so, we lay out some assumptions for $v(x,t)$ and $M(x)$. 

\begin{itemize}
    \item[(A1)] $M(x)$ is $C^2$-smooth with bounded, Lipschitz continuous Hessian $\mathcal H_M$;

    \item[(A2)] $v(x,t)$ is $C^1$-smooth, bounded from above and below by positive constants, has bounded, Lipschitz continuous gradient.
\end{itemize}

We can then bound the error in the approximation \eqref{eq:proxApprox} as follows. 

\begin{theorem}
Fix $p,\nu \in \R^n$, $t,\Delta t>0$. Assume that $M(x)$ and $v(x,t)$ satisfy (A1) and (A2) respectively, and take $H(x,p,t)$ as defined in \eqref{eq:HamVec}. Then for $\tau >0$ small enough, there exists $x \in \R^n$ such that $$x = \prox_{-\tau \Delta t H(\,\cdot,p,t) }(\nu)$$ is uniquely determined and satisfies \begin{equation} \label{eq:proxIneq} \abs{x- \nu} \le \tau \Delta t\abs{p}\Big(\abs{\nabla v(x,t)} + 2\abs{v(x,t)}\|{\mathcal H}_M(x)\| \Big).\end{equation}
\end{theorem}

\begin{remark}
In essence, \eqref{eq:proxIneq} states that as long as $p^{k+1}_j$ remains bounded, assumptions (A1) and (A2) ensure that, all else remaining equal, the approximation of $x^{k+1}_j$ given by \eqref{eq:proxApprox} can be made arbitrarily accurate by decreasing $\tau$ and/or $\Delta t$. In particular, by decreasing $\tau$, one may ensure that the error in this approximation is on the order of $\Delta t$, the same as the error induced by the backward Euler time discretization used in the derivation of \eqref{eq:generalizedHopfLax}. We note that under (A1) and (A2), \eqref{eq:proxIneq} can be written $\abs{x-\nu}\le C\tau\Delta t \abs{p}$ if desired. Also, it is likely that the assumptions could be relaxed slightly (for example, to allow for $v,\nabla v,\nabla M$ and $\mathcal H_M$ to be merely bounded on compact sets); for brevity, we only include this strong set of sufficient conditions.
\end{remark}

\begin{proof} 
If $p= 0$, the result is trivial since $x = \nu$ is the minimizer. We assume $p \neq 0$. Writing out the proximal operator for our Hamiltonian and ignoring constants, we see \begin{equation} \label{eq:proxx}x = \argmin_{x^* \in \R^n} \left\{-\tau\Delta t v(x^*,t) \sqrt{p^TA(x^*)p} + \frac 1 2 \abs{x^*-\nu}^2\right\}.\end{equation} Given (A1) and (A2), the function $-v(x,t)\sqrt{p^TA(x)p}$ is semiconvex with a linear modulus \cite[Proposition 2.1.2]{cannarsa2004semiconcave}, meaning that for $\tau$ small enough the function being minimized in \eqref{eq:proxx} is strongly convex and there is a unique minimizer. The minimum in \eqref{eq:proxx} occurs at a point where the gradient of the function being minimized is zero. Using \eqref{eq:gradA} and once again suppressing the argument of $A$ and $M$, this shows that \begin{align*} x-\nu &= \tau \Delta t \sqrt{p^TAp} \nabla v(x,t) \\ &\hspace{0.3cm} + \tau \Delta t v(x,t) \frac{(p^T\nabla M)^2{\mathcal H}_M\nabla M - (1+\abs{\nabla M}^2)(p^T\nabla M){\mathcal H}_Mp}{\sqrt{p^TAp}\left(1+\abs{\nabla M}^2\right)^2}\end{align*} We recall that the matrix $A(x)$ defined in \eqref{eq:Amat} is SPD with eigenvalues $1/(1+\abs{\nabla M(x)}^2)$ and $1$. Thus $$\frac{\abs{p}^2}{1+\abs{\nabla M}^2} \le p^TAp \le \abs{p}^2.$$ Using this inequality, along with the Cauchy-Schwarz inequality, and the obvious inequality $\abs{\nabla M} \le (1+\abs{\nabla M}^2)^{1/2}$, we see \begin{align*}
    \abs{x-\nu} &\le \tau\Delta t \abs p \abs{\nabla v(x,t)}  \\ &\hspace{0.7cm} +\tau \Delta t \abs{v(x,t)}\frac{\abs{p}^2 \abs{\nabla M}^3 \|\mathcal H_M\| + \abs{p}^2(1+\abs{\nabla M}^2)\abs{\nabla M}\|\mathcal H_M\|}{\frac{\abs{p}}{(1+\abs{\nabla M}^2)^{1/2}}(1+\abs{\nabla M}^2)^2 }\\
    &\le \tau\Delta t \abs{p} \Big(\abs{\nabla v(x,t)} + 2\abs{v(x,t)}\|\mathcal H_M(x)\|\Big)
\end{align*} as desired. 
\end{proof} We note that this result neglects to include the indicator function $\mathds 1_{x\neq x_f}(x)$. However, if this is approximated by \eqref{eq:indApprox}, the result still holds since the approximating function $1-e^{-B\abs{x-x_f}^2}$ is smooth with exponentially decaying derivatives, though these derivatives will appear in the estimate \eqref{eq:proxIneq}. 

Chambolle and Pock \cite{Chambolle2011AFP} prove convergence of the algorithm under conditions that can be translated to our scenario as follows: $H(x,p,t)$ is state-independent, $\kappa=1$, and $\sigma \tau \|L(x)^{-1}\|^2\|D\|^2 < 1$ where $L(x)L(x)^T=A(x)$ and $D$ is a $(J+1)n \times (J+1)n$ block matrix with $I_{n\times n}$ the along the diagonal, and $-I_{n\times n}$ along the first subdiagonal (this comes from \eqref{eq:SPP}). In \cite{lin2018splitting,parkinson2024efficient,ParkPolage}, there is good empirical evidence that convergence can still be accomplished under similar conditions for state-dependent Hamiltonians, though a proof is elusive. Using these conditions, as a starting point, we will always set $\kappa = 1$. We note that $\|D\|\le 2$ and $\|L(x)^{-1}\|^2 = 1+\abs{\nabla M(x)}^2$, the maximum eigenvalue of $A(x)^{-1}$. 
Thus at the very least one should require that \begin{equation} \label{eq:sigtau}\sigma \tau < \frac{1}{4\max\Big(1+\abs{\nabla M(x)}^2\Big)}. \end{equation} In light of \eqref{eq:proxIneq}, in practice we fix $\sigma =1$ and decrease $\tau$ to satisfy \eqref{eq:sigtau} while also ensuring that the approximation of $x^{k+1}_j$ is suitably accurate. 

We include a few final implementation notes. When using the approximation \eqref{eq:proxApprox}, we found that Algorithm \ref{Algorithm1} converges much faster, both in terms of iteration count and clock time, but the resulting path is often somewhat jagged, though it is essence the correct shape. Accordingly, after initializing the algorithm randomly, we use the approximation \eqref{eq:proxApprox} to resolve $x^{k+1}_j$  for the first 2000 iterations so as to resolve the basic skeleton of the path. After this, we switch to resolving $x^{k+1}_j$ via the gradient descent approximation \eqref{eq:actualGradDesc} for the remaining iterations. When we do this, we approximate the indicator function $\mathds 1_{x\neq x_f}(x)$ via \eqref{eq:indApprox}. Recall, the cost functional we are minimizing is \eqref{eq:costFunc}. This will assign infinite cost to any path which does not reach the desired endpoint $x_f$, so the traveler must reach the endpoint. The indicator function then incentivizes the traveler to reach the endpoint \emph{as quickly as possible}, as opposed to lollygagging for a bit and then heading toward the endpoint. The constant $B$ in \eqref{eq:indApprox} could then be seen as a measure of this incentivization: smaller $B$ means the traveler has less incentive to reach $x_f$ quickly, whereas $B\to \infty$ would recover the perfect incentivization to reach $x_f$ as efficiently as possible. Accordingly, we would like to take $B$ very large. However, when $B$ is very large, the gradient descent approximation \eqref{eq:actualGradDesc} requires smaller gradient descent step $\eta$, and this diminishes the effects of $v(x,t)$ and $M(x)$. So as a compromise, we begin the gradient descent with a fairly large descent rate ($\eta = 0.025$ in the below examples) and a fairly small value $B$ ($B = 50$), and every 1000 iterations thereafter, we halve the gradient descent rate and increase $B$ by 50. For the random initialization, we always initialize $\{x^0_j\}_{j=1}^J$ to be a straight-line path between $x$ and $x_f$ and then add noise distributed like $\mathcal N(0,0.1)$ to each component; for $\{p^0_j\}_{j=1}^J$, each component is initialized randomly with distribution $\mathcal N(0,0.1)$.

One last note is that the algorithm is embarrassingly parallelizable with respect to computation of individual paths (or equivalently individual value $u(x,t)$ of the value function), so if one does desire to resolve the value function in the entirety of a spatial domain, a parallel implementation is possible. 

\section{Results \& Discussion}

We include a few example to demonstrate the effectiveness and efficiency of our algorithm. All simulations were run in MATLAB on the second author's desktop computer with an Intel(R) Core(TM) i7-14700 processor running at 2.1GHz with 32GB of RAM. The code used to generate all the ensuing figures is uploaded to GitHub\footnote{\url{https://github.com/chparkinson/path_plan_manifold_Hopf_Lax/}}. In all simulations, we set the maximum iteration count to 40000. None of the paths displayed failed to resolve to within a tolerance of $10^{-3}$ within this iteration count, and most required significantly fewer. In the two dimensional example, we hard-code the Cholesky factorization of the matrix $A(x)$ given in \eqref{eq:chol}, while in the high dimensional example we use MATLAB's built in Cholesky factorization routine. In all examples, we choose $\Delta t = 0.1$. The value of $t$ changes from example to example, but is somewhat arbitrary as long as it is large enough that there is time to reach the end point.

\begin{figure}[b!]
    \centering
    \begin{tabular}{l l l}
        $\abs{x}$ & $u$ & Err. \\ 
\hline
1.9118  &  1.9344  & 2.2599e-02 \\ 
1.7417  &  1.7420  & 2.4852e-04 \\ 
1.7157  &  1.7528  & 3.7081e-02 \\ 
1.8538  &  1.8602  & 6.3243e-03 \\ 
1.7823  &  1.7723  & 1.0003e-02\\ 
\hline
    \end{tabular} \,\,\,\,\,\,
    \begin{tabular}{l l l}
    $\abs{x}$ & $u$ & Err. \\ 
\hline
    2.5188  &  2.5392  & 2.0418e-02 \\ 
1.2081  &  1.2005  & 7.5970e-03 \\ 
1.2721  &  1.2732  & 1.1581e-03 \\ 
1.9907  &  1.9877  & 3.0218e-03 \\ 
1.7064  &  1.7140  & 7.6603e-03 \\ 
\hline
    \end{tabular}
    \caption{The approximate solution $u$ provided by Algorithm \ref{Algorithm1} in the case that $M \equiv 0, v \equiv 1$, and $x_f = (0,0,\ldots, 0)$. The state space is $\R^{10}$ and the $x$ values are 10 randomly generated members of $[-1,1]^{10}.$ In this case, the exact solution is $u(x,t) = \abs{x}.$ Note that our approximation error is never larger than $3.8 \times 10^{-2}$, even though $\Delta t = 0.1$.}
    \label{fig:tab}
\end{figure}

In a preliminary example to demonstrate the accuracy of our approximation, we take $M(x) \equiv 0$, $v\equiv 1$ and $x_f = (0,0,\ldots,0).$ In this case, all optimal paths are straight lines and for any $t \ge \abs{x}$, we have $u(x,t) = \abs{x}$ since this is the time required to travel from $x$ to the origin. We solve this problem numerically with $n = 10$ spatial dimensions, for ten randomly chosen values of $x$ and report the error in our approximation in Figure \ref{fig:tab}. Note that the largest error is roughly $10^{-2}$ even though $\Delta t = 0.1$.

The first visual example involves path planning on the manifold $(x,y,M(x,y))$ where $M(x,y) = a\sin(\pi x)\cos(\pi y)$ for some constant $a$. In this example, the velocity is constant $v \equiv 1$, so we are simply finding geodesics on the manifold. Here we fix the end point $x_f = (1,1)$, randomly generate 20 starting points in $[-1,1] \times [-1, 1]$, and compute the optimal path from each starting point. The results of this simulation are displayed in the top panel of Figure \ref{fig:1}. For these simulations, Algorithm \ref{Algorithm1} converged in an average of 1.04 seconds CPU time and 6290 iterations. In the bottom panels of Figure \ref{fig:1}, we plot the optimal path along the manifold from $x = (-1,-1)$ to $x_f = (1,1)$. For the plot in the bottom left panel, we set $a = 1$, so the ``hills" are rather gentle, whereas in the bottom right panel, we set $a = 3$ so that the hills are much steeper. As the particle travels along the manifold, the distance traveled is locally $(1+\abs{\nabla M(x,y)}^2)^2$ units per unit distance in the $xy$-plane, so we expect that as the terrain becomes steeper the particle is more strongly incentivized to travel along the flat areas, and we see this. With the gentler hills, the particle cuts corners more liberally, whereas with the steeper hills, the path more closely sticks to the straight lines between the hills and valleys. 

\begin{figure}
\centering
\fbox{\includegraphics[width=0.31\textwidth,trim = 50 10 35 10,clip]{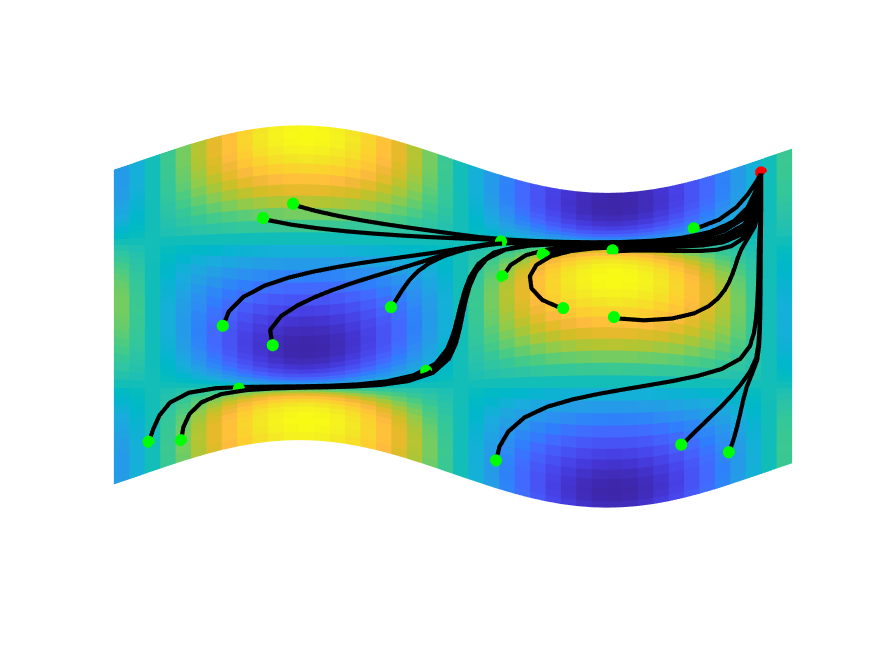}} \\ 
\fbox{\includegraphics[width=0.31\textwidth,trim = 50 10 35 10,clip]{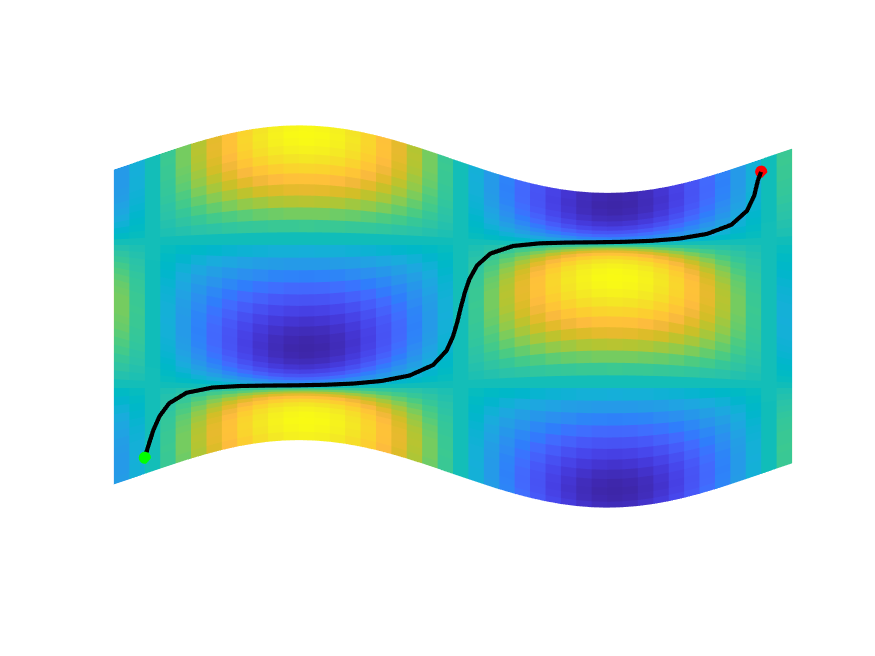}} 
\fbox{\includegraphics[width=0.31\textwidth,trim = 50 10 35 10,clip]{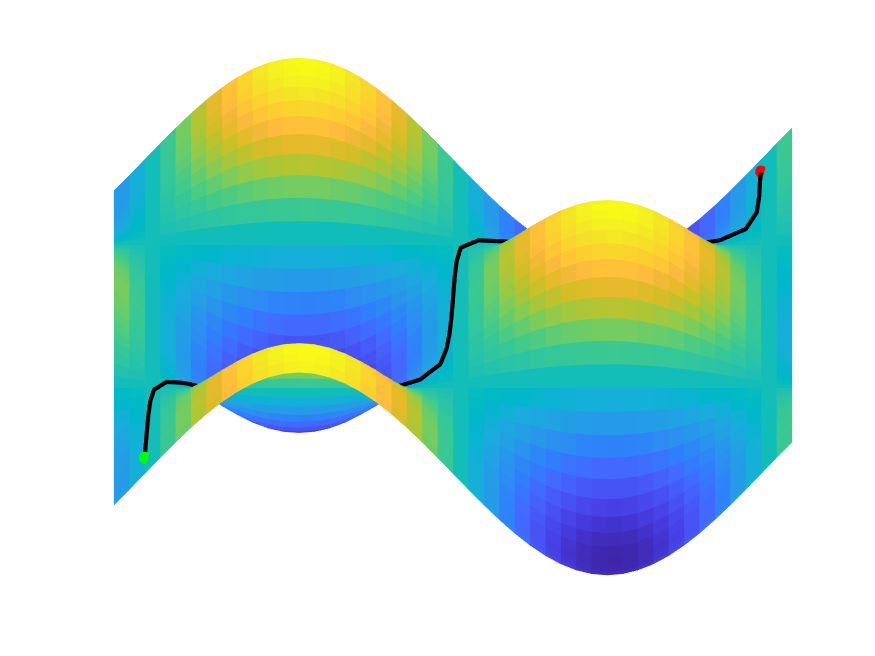}}
\caption{Optimal paths on the graph of $M(x,y) = a\sin(\pi x)\cos(\pi y)$ for constant $a$. Top panel: $a=1$, optimal paths are plotted from 20 randomly chosen points in $[-1,1]\times [-1,1]$ to the point $(1,1)$. Bottom left panel: $a =1$, plotted is the optimal path from $(-1,-1)$ to $(1,1)$. Bottom right panel: $a = 3$, plotted is the optimal path from $(-1,-1)$ to $(1,1)$. Notice that when $a$ is larger, the path is more apt to stay in the flat regions, avoiding the hills and valleys, since travel distance along the manifold is $(1+\abs{\nabla M(x,y)}^2)^{1/2}$ per unit distance in the $xy$-plane.}
\label{fig:1}
\end{figure}

In the second example, our manifold is the graph of $M(x,y) = 2\text{exp}(-(x^2+y^2))$. The results are in Figure \ref{fig:2}. Again, we set $x_f = (1,1)$, and in this case we take randomly chosen initial points in $[-0.85,-0.65]\times [-0.85,-0.65].$  Here the manifold is like a steep hill between the initial points and the final point. Optimal paths should walk around the hill. Because the manifold and the choice of initial points is symmetric about the line $y = x$, when the velocity is constant $v \equiv 1$, we expect that about half of the paths should choose to walk around the hill in either direction (depending on which side of the line $y=x$ the initial point lies on). This is exactly what we find in the left panel of Figure \ref{fig:2}. If instead the velocity is $v(x,y) = 1+(x-1)^2$ as in the right panel of Figure \ref{fig:2}, then the velocity is much slower on the right side of the image, so regardless of the initial position, if it is near enough to the line $y = x$, the optimal action is to ``stay left" as long as possible, which is why all of the paths in that image travel up in $y$ before traveling right in $x$. In this example, Algorithm \ref{Algorithm1} resolved the paths in the left panel in an average of 1.17 seconds of CPU time and 9439 iterations, and paths in the right panel in an average of 0.92 seconds of CPU time and 8442 iterations.  

\begin{figure}[t!]
\centering
\fbox{\includegraphics[width=0.4\textwidth,trim = 50 30 35 60,clip]{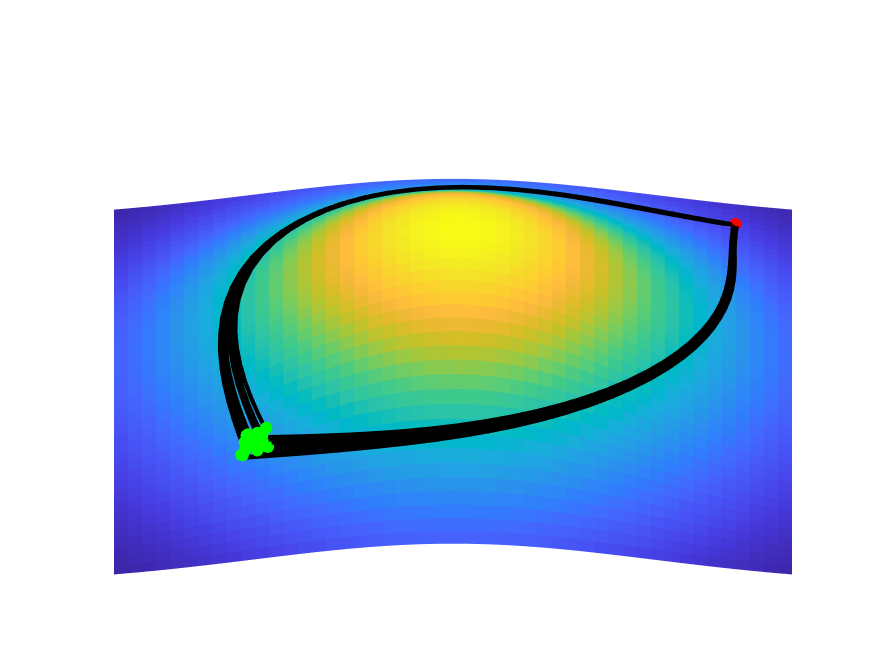}}\,\,
\fbox{\includegraphics[width=0.4\textwidth,trim = 50 30 35 60,clip]{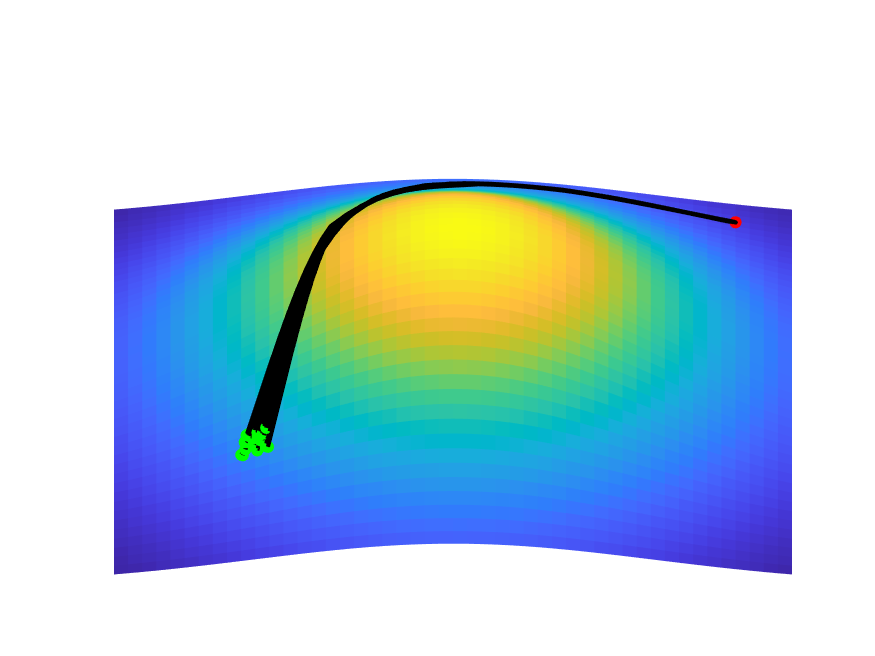}}
\caption{Optimal paths on the manifold $M(x,y) = 2\text{exp}(-(x^2+y^2))$. The final point is $x_f = (1,1)$ in each image, and there are 20 initial points are randomly selected in $[-0.85,-0.65]\times [-0.85,-0.65]$ (the random selection for each image is done separately). In the left panel, the velocity is $v \equiv 1$, so that, as one may expect, roughly half of the paths travel either direction around the large hill in the middle. In the right panel, the velocity is $v(x,y) = 1+(x-1)^2$, meaning much faster travel is allowed in the left hand side of the image. Because of this, each particle travels up in $y$ before traveling $right$ in $x$.}
\label{fig:2}
\end{figure}

In our third example, we demonstrate that Algorithm \ref{Algorithm1} can generate optimal trajectories efficiently even in high dimensional problems. Accordingly, we again use the manifold $M(x) = 2\text{exp}(-\abs{x}^2)$, and $v \equiv 1$, but now test the algorithm in dimensions $10-30.$ In this case, $x_f = (1,1,1,\ldots,1)$ and $x = (-0.9,-1,-1,\ldots,-1).$ We choose $-0.9$ for the first coordinate of $x$ to ensure that there is a unique optimal path, instead of two equally optimal paths traversing opposite sides of the Gaussian ``mountain." For each dimension $n = 10,11,\ldots, 30$, we run 10 trials with different random initializations and record the average CPU time as a function of the dimension. The results---as well as a plot of each component $x_j(t)$ of the optimal path in 25 dimensions---are included in Figure \ref{fig:3}. In the 25 dimensional example, if the path was a straight line, all components would be plotted as straight lines from $x_j(0)$ at time $0$ to $1$ at time $t \approx 11.30$ (which is the optimal travel time resolved by the algorithm). For the optimal path, we note that the $x_1(t)$ which started at $-0.9$ deviates far from the ``straight line" path. This is to allow the other components (for whom the problem is entirely symmetric) to take a path in 24 dimensional space which more closely resembles a straight line. This behavior should be expected: in 25 dimensions, it is costly to move all entries away from the center to avoid the Gaussian in the middle. It is much easier to let $x_1(t)$ take a roundabout path, while letting the others travel in a straight line. In the plot of CPU time as a function of dimension, we note that even in 30 dimensions, Algorithm \ref{Algorithm1} is able to resolve optimal paths in in average of 2.87 seconds of CPU time with a standard deviation of roughly 0.34 seconds over the 10 trials. Problems in such high dimensions are entirely intractable for classical PDE-based path planning methods which rely on grid-based approximations of solutions of PDE, because the complexity of the grid-based solvers scale exponentially with dimension. While it is hard to precisely pin down what the dimensional scaling is here, the only discretization is in time, meaning that, on the face of it, assuming the minimization problems can be resolved exactly and the Cholesky factorization is not the most expensive operation, individual iterations in Algorithm \ref{Algorithm1} scale linearly with dimension (simply due to the increased FLOPs), and linearly with increases in $t$ and decreases in $\Delta t$. Some complicating factors are that larger travel time $t$ is required in higher dimensions because there is more distance to cover, iterations required for convergence may depend on dimension, and in higher dimensions, the Cholesky factorization (which has a complexity of $O(n^3)$ in dimension $n$) will become very expensive. So while Figure \ref{fig:3} may not definitively demonstrate linear scaling with dimension (though it does seem to be approximately linear for dimensions 10-30), it at least seems to indicate the scaling is much better than exponential. In any case, up to at least dimension 30 the algorithm is efficient enough to compute optimal paths in near real time.

\begin{figure}[t!]
    \centering
    \includegraphics[width=0.45\textwidth]{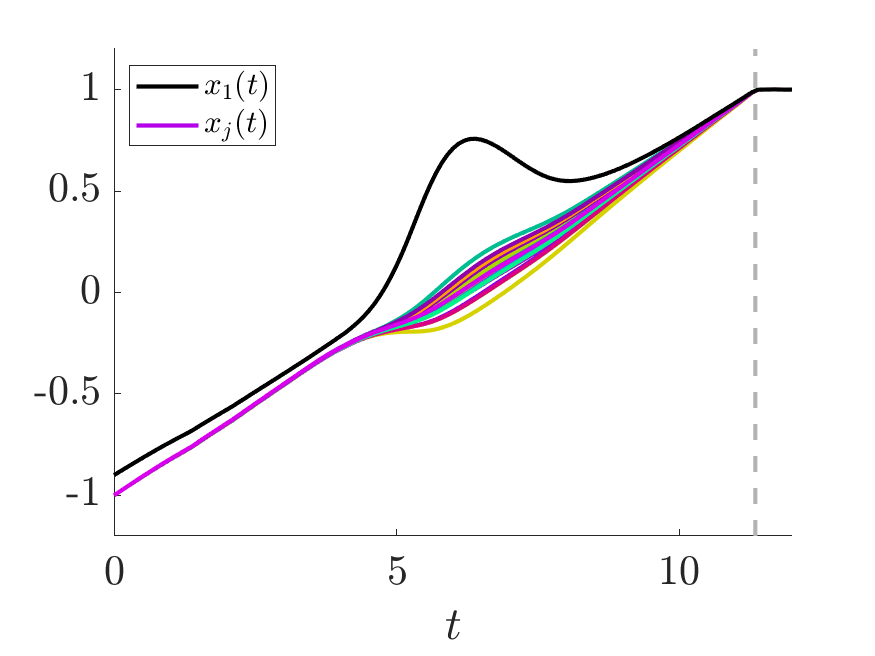} \,\,
    \includegraphics[width=0.45\textwidth]{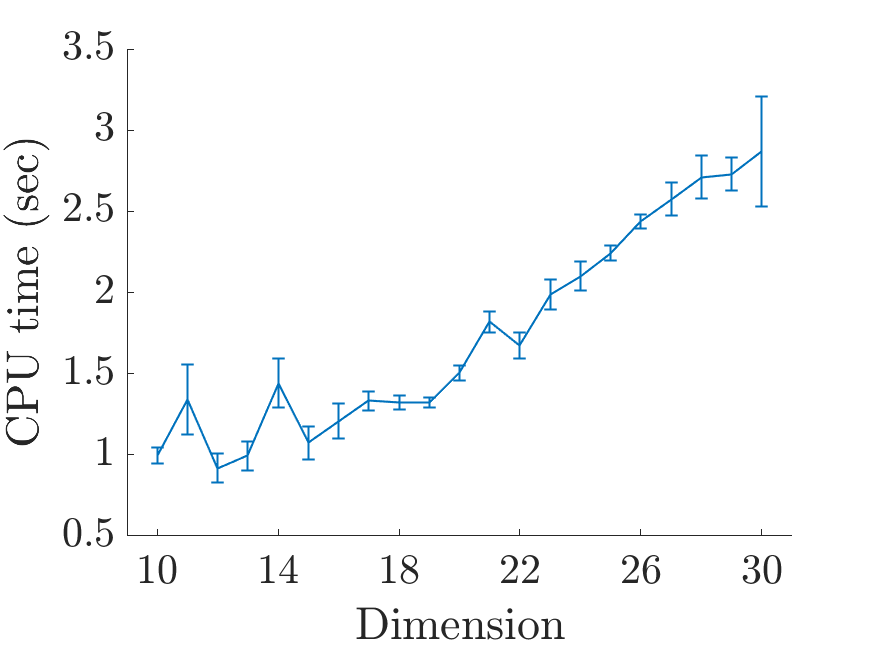} \,\,
    \caption{Left panel: an optimal path from $x = (-0.9,-1,-1,\ldots,-1)$ to $x_f = (1,1,1,\ldots,1)$ on the graph of $M(x) = \text{exp}(-\abs{x}^2)$ in 25 dimensions. Plotted are each coordinate of the solution. The gray dotted line is the value $u(x,t) \approx 11.30$ resolved by the algorithm, which is an approximation of the optimal travel time. We see that all coordinates do indeed reach $1$ at roughly that time and then stay there. Right panel: the average CPU time required to resolve the optimal path as a function of the dimension. Averages are taken over ten trials and the bars represent the standard deviation in the results. These results seem to indicate that the dimensional scaling for the algorithm is at very least subexponential, and looks approximately linear for the dimensions 10-30.}
    \label{fig:3}
\end{figure}

\section{Conclusion} In this manuscript, we present an algorithm for PDE-based optimal path planning on manifolds. The algorithm relies on a dynamic programming and Hamilton-Jacobi-Bellman approach to optimal control, and exploits a discrete Hopf-Lax type formula that allows one to approximately solve the Hamilton-Jacobi-Bellman equation at individual points without needing to discretize the spatial domain and resort to finite difference approximations. We propose a PDHG inspired algorithm to solve the requisite saddle-point problem and resolve optimal trajectories in both state and co-state space. We demonstrate the efficiency of our method on a few example problems, including problems in dimension up to 30. 

There are several avenues for future work to explore. To actually prove that the conjectured Hopf-Lax formula can represent the value function for a large class of HJB equations would be a large theoretical leap forward. On the practical side, adaptation of this and similar algorithms for real time implementation (and the more general development of PDE-based path planning methods like this) could provide a scalable, interpretable, and efficient method for real time path planning in any number of scenarios. In fact, with slight tweaks similar methods are likely possible for manifolds of higher codimension or manifolds which are not the graphs of smooth functions. We are also interested to adapt the method presented here to differential games.

\appendix
\section{Derivation of the Hamiltonian} \label{nHJBDeriv} In this appendix, we derive the Hamilton-Jacobi-Equation for our control problem. Beginning from \eqref{eq:Hamiltonian1}, we see that for fixed $p$ and $\g = \nabla M$, this boils down to computing \begin{equation} \label{eq:ndimminprob}\inf_{a \in \mathbb S^{n-1}} f(a) \,\,\,\, \text{ where} \,\,\, f(a) = \frac{\innerprod p a}{\sqrt{1+\innerprod {\g}{a}^2}}.\end{equation} Since we are minimizing a continuous function on a compact set, we know a minimizer exists. If $p,\g$ are linearly dependent, the proof is significantly simplified be writing one as a constant multiple of the other; we omit this case and assume that $p, \g$ are linearly independent. By the method of Lagrange multipliers, the minimizer must occur at a unit vector $a$ such that $\nabla f(a) = 2\lambda a$ for some $\lambda$. Writing this equation out yields \begin{equation} \label{eq:lagMult1} \frac{(1+\innerprod{\g}{a}^2)p - \innerprod{p}a\innerprod\g a \g }{(1+\innerprod{\g}{a})^{3/2}} = 2\lambda a.\end{equation} Taking the inner product with $a$ and using $\abs a^2 = 1$ gives $$2\lambda = \frac{\innerprod{p}{a}}{(1+\innerprod{\g}{a}^2)^{3/2}}.$$ Inserting this into \eqref{eq:lagMult1} and clearing the denominator shows that \begin{equation}\label{eq:lagMult2} \innerprod{p}a a = (1+\innerprod{\g}a^2)p - \innerprod{p}a\innerprod{\g}a \g. \end{equation} Equation \eqref{eq:lagMult2} shows that $a$ lies in the span of $p$ and $\g$, whereupon we can write $a = k_1 p + k_2 \g$. Inserting this representation of $a$ into \eqref{eq:lagMult2}, we can solve for $k_1,k_2$ in terms of $p$ and $\g$.  We see $$
k_1 \innerprod p a p + k_2 \innerprod p a \g = (1+\innerprod{\g}a^2)p - \innerprod{p}a\innerprod{\g}a \g. $$ By linear independence, the coefficients must be equal, so \begin{equation} \label{eq:k2} k_2 = - \innerprod \gamma a = -\innerprod p \g k_1 - \abs{\g}^2k_2  \,\,\, \implies \,\,\, k_2 = -\frac{\innerprod p \g}{1+\abs{\g}^2}k_1.\end{equation} Then $$k_1 = \frac{1+k_2^2}{\innerprod p a} = \frac{1+k_2^2}{\abs{p}^2k_1 + \innerprod p \g k_2} = \frac{1 + \frac{\innerprod p \g ^2}{(1+\abs{\g}^2)^2}k_1^2}{\left(\abs{p}^2 - \frac{\innerprod p \g^2}{1+\abs{\g}^2}\right)k_1}.$$ Solving yields \begin{equation} \label{eq:k1}k_1^2 = \frac{\left(1+\abs{\g}^2\right)^2}{\abs{p}^2\left(1+\abs \g^2\right)^2 - \innerprod p \g^2\left(2+\abs\g^2\right)}.\end{equation} Along with \eqref{eq:k2}, this shows that the minimizer is one of the vectors \begin{equation}\label{eq:k1k2}
a^* = \pm \frac{\left(1+\abs{\g}^2\right)p - \innerprod p \g \g}{\sqrt{\abs{p}^2\left(1+\abs \g^2\right)^2 - \innerprod p \g^2\left(2+\abs\g^2\right)}}.
\end{equation} \emph{Post hoc}, it is obvious that the minus sign leads to a minimum and the plus leads to a maximum, so we take the minus sign. Denoting the denominator in \eqref{eq:k1k2} by $\delta$ momentarily and inserting \eqref{eq:k1k2} into \eqref{eq:ndimminprob}, we see that \begin{align*}
    f(a^*) &= - \frac{(1+\abs \g^2) \abs p^2 - \innerprod{p}\g^2 }{\delta\sqrt{1 + \frac 1 {\delta^2}\left((1+\abs \g^2)\innerprod p \g - \innerprod p \g \abs \g^2\right)^2}} \\ 
    &= - \frac{(1+\abs \g^2) \abs p^2 - \innerprod{p}\g^2 }{\sqrt{\delta^2  +\innerprod p \g^2}} \\
    &= -\frac{(1+\abs \g^2) \abs p^2 - \innerprod{p}\g^2 }{\sqrt{\abs{p}^2\left(1+\abs \g^2\right)^2 - \innerprod p \g^2\left(2+\abs\g^2\right) + \innerprod p \g^2}} \\ 
    &= -\frac{(1+\abs \g^2) \abs p^2 - \innerprod{p}\g^2 }{\sqrt{\abs{p}^2\left(1+\abs \g^2\right)^2 - \innerprod p \g^2\left(1+\abs\g^2\right)}}\\
    &= -\sqrt{\frac{(1+\abs \g^2) \abs p^2 - \innerprod{p}\g^2 }{1 + \abs{\g}^2}} \\
    & = -\sqrt{\abs{p}^2 - \frac{\innerprod{p}\g^2}{1+\abs{\g}^2}} = -\sqrt{p^T \left(I - \frac{\g\g^T}{1+\abs{\g}^2}\right)p}.
\end{align*} Replacing $\g = \nabla M$ and inserting this result into \eqref{eq:Hamiltonian1}, we arrive at \eqref{eq:Hamiltonian} as desired.

\bibliographystyle{plain}
\bibliography{references}

\begin{thebibliography}{10}

\bibitem{Arnold1}
David~J Arnold, Dayne Fernandez, Ruizhe Jia, Christian Parkinson, Deborah
  Tonne, Yotam Yaniv, Andrea~L Bertozzi, and Stanley~J Osher.
\newblock Modeling environmental crime in protected areas using the level set
  method.
\newblock {\em SIAM Journal on Applied Mathematics}, 79(3):802--821, 2019.

\bibitem{beck}
Amir Beck.
\newblock {\em First-order methods in optimization}.
\newblock SIAM, 2017.

\bibitem{Bellman1952}
Richard Bellman.
\newblock {On the Theory of Dynamic Programming}.
\newblock {\em Proceedings of the National Academy of Sciences},
  38(8):716--719, 1952.

\bibitem{Bertsekas}
Dimitri Bertsekas.
\newblock {\em Dynamic programming and optimal control: Volume I}, volume~4.
\newblock Athena scientific, 2012.

\bibitem{Borquez}
Javier Borquez, Shuang Peng, Yiyu Chen, Quan Nguyen, and Somil Bansal.
\newblock {Hamilton-Jacobi Reachability Analysis for Hybrid Systems with
  Controlled and Forced Transitions}.
\newblock In {\em Proceedings of Robotics: Science and Systems}, Delft,
  Netherlands, July 2024.

\bibitem{Bryson}
Arthur~Earl Bryson.
\newblock {\em Applied optimal control: optimization, estimation and control}.
\newblock Routledge, 2018.

\bibitem{cannarsa2004semiconcave}
Piermarco Cannarsa and Carlo Sinestrari.
\newblock {\em Semiconcave functions, Hamilton-Jacobi equations, and optimal
  control}, volume~58.
\newblock Springer Science \& Business Media, 2004.

\bibitem{Cartee}
Elliot Cartee and Alexander Vladimirsky.
\newblock Control-theoretic models of environmental crime.
\newblock {\em SIAM Journal on Applied Mathematics}, 80(3):1441--1466, 2020.

\bibitem{Chambolle2011AFP}
A.~Chambolle and Thomas Pock.
\newblock {A First-Order Primal-Dual Algorithm for Convex Problems with
  Applications to Imaging}.
\newblock {\em Journal of Mathematical Imaging and Vision}, 40:120--145, 2011.

\bibitem{Bohan}
Bohan Chen, Kaiyan Peng, Christian Parkinson, Andrea~L Bertozzi, Tara~Lyn
  Slough, and Johannes Urpelainen.
\newblock Modeling illegal logging in brazil.
\newblock {\em Research in the Mathematical Sciences}, 8(2):29, 2021.

\bibitem{ChoiBansal}
Jason~J Choi, Ayush Agrawal, Koushil Sreenath, Claire~J Tomlin, and Somil
  Bansal.
\newblock Computation of regions of attraction for hybrid limit cycles using
  reachability: An application to walking robots.
\newblock {\em IEEE Robotics and Automation Letters}, 7(2):4504--4511, 2022.

\bibitem{CHOW2019376}
Yat~Tin Chow, Jérôme Darbon, Stanley Osher, and Wotao Yin.
\newblock {Algorithm for overcoming the curse of dimensionality for
  state-dependent Hamilton-Jacobi equations}.
\newblock {\em Journal of Computational Physics}, 387:376--409, 2019.

\bibitem{Visc1}
Michael~G Crandall and Pierre-Louis Lions.
\newblock Viscosity solutions of hamilton-jacobi equations.
\newblock {\em Transactions of the American mathematical society},
  277(1):1--42, 1983.

\bibitem{darbon2016}
Jérôme Darbon and Stanley Osher.
\newblock {Algorithms for Overcoming the Curse of Dimensionality for Certain
  Hamilton-Jacobi Equations Arising in Control Theory and Elsewhere}, 2016.

\bibitem{Dong}
Qiao-Li Dong, Yeol~Je Cho, Songnian He, Panos~M Pardalos, and Themistocles~M
  Rassias.
\newblock {\em The Krasnosel'ski{\u\i}-Mann Iterative Method: Recent Progress
  and Applications}.
\newblock Springer, 2022.

\bibitem{Dubins1957}
L.~E. Dubins.
\newblock On curves of minimal length with a constraint on average curvature,
  and with prescribed initial and terminal positions and tangents.
\newblock {\em American Journal of Mathematics}, 79(3):497--516, 1957.

\bibitem{duits2018optimal}
Remco Duits, Stephan~PL Meesters, J-M Mirebeau, and Jorg~M Portegies.
\newblock Optimal paths for variants of the 2d and 3d reeds--shepp car with
  applications in image analysis.
\newblock {\em Journal of Mathematical Imaging and Vision}, 60:816--848, 2018.

\bibitem{EvansControl}
Lawrence~C Evans.
\newblock An introduction to mathematical optimal control theory version 0.2.

\bibitem{EvansPDE}
Lawrence~C Evans.
\newblock {\em Partial differential equations}, volume~19.
\newblock American Mathematical Society, 2022.

\bibitem{FlemingRishel}
Wendell~H Fleming and Raymond~W Rishel.
\newblock {\em Deterministic and stochastic optimal control}, volume~1.
\newblock Springer Science \& Business Media, 2012.

\bibitem{Vlad1}
Marissa Gee and Alexander Vladimirsky.
\newblock Optimal path-planning with random breakdowns.
\newblock {\em IEEE Control Systems Letters}, 6:1658--1663, 2021.

\bibitem{Hu2015}
Chia-Yu Hu and Chun-Hung Lin.
\newblock Reverse ray tracing for transformation optics.
\newblock {\em Opt. Express}, 23(13):17622--17637, Jun 2015.

\bibitem{kaoosher2005}
Chiu-Yen Kao, Stanley Osher, and Yen-Hsi Tsai.
\newblock {Fast Sweeping Methods for Static Hamilton-Jacobi Equations}.
\newblock {\em SIAM Journal on Numerical Analysis}, 42(6):2612--2632, 2005.

\bibitem{Krano}
Mark~Aleksandrovich Krasnosel'skii.
\newblock Two remarks on the method of successive approximations.
\newblock {\em Uspekhi matematicheskikh nauk}, 10(1):123--127, 1955.

\bibitem{KULATHUNGA2022152}
Geesara Kulathunga.
\newblock A reinforcement learning based path planning approach in 3d
  environment.
\newblock {\em Procedia Computer Science}, 212:152--160, 2022.
\newblock 11th International Young Scientist Conference on Computational
  Science.

\bibitem{Liberzon}
Daniel Liberzon.
\newblock {\em Calculus of variations and optimal control theory: a concise
  introduction}.
\newblock Princeton university press, 2011.

\bibitem{lin2018splitting}
Alex~Tong Lin, Yat~Tin Chow, and Stanley~J Osher.
\newblock A splitting method for overcoming the curse of dimensionality in
  hamilton--jacobi equations arising from nonlinear optimal control and
  differential games with applications to trajectory generation.
\newblock {\em Communications in Mathematical Sciences}, 16(7):1933--1973,
  2018.

\bibitem{liu2021acceleration}
Yanli Liu, Yunbei Xu, and Wotao Yin.
\newblock Acceleration of primal--dual methods by preconditioning and simple
  subproblem procedures.
\newblock {\em Journal of Scientific Computing}, 86(2):21, 2021.

\bibitem{luo2016convergence}
Songting Luo and Hongkai Zhao.
\newblock Convergence analysis of the fast sweeping method for static convex
  hamilton--jacobi equations.
\newblock {\em Research in the Mathematical Sciences}, 3(1):35, 2016.

\bibitem{Mann}
W~Robert Mann.
\newblock Mean value methods in iteration.
\newblock {\em Proceedings of the American Mathematical Society},
  4(3):506--510, 1953.

\bibitem{mantegazza2002hamilton}
Carlo Mantegazza and Andrea~Carlo Mennucci.
\newblock Hamilton-jacobi equations and distance functions on riemannian
  manifolds.
\newblock {\em arXiv preprint math/0201296}, 2002.

\bibitem{Vlad2}
Cole Miles and Alexander Vladimirsky.
\newblock Stochastic optimal control of a sailboat.
\newblock {\em IEEE Control Systems Letters}, 6:2048--2053, 2021.

\bibitem{levelSet}
Stanley Osher and James~A Sethian.
\newblock Fronts propagating with curvature-dependent speed: Algorithms based
  on hamilton-jacobi formulations.
\newblock {\em Journal of computational physics}, 79(1):12--49, 1988.

\bibitem{parkinson2021rotating}
Christian Parkinson.
\newblock A rotating-grid upwind fast sweeping scheme for a class of
  hamilton-jacobi equations.
\newblock {\em Journal of Scientific Computing}, 88(1):13, 2021.

\bibitem{SteepTerrain2}
Christian Parkinson, David Arnold, Andrea Bertozzi, and Stanley Osher.
\newblock A model for optimal human navigation with stochastic effects.
\newblock {\em SIAM Journal on Applied Mathematics}, 80(4):1862--1881, 2020.

\bibitem{SteepTerrain1}
Christian Parkinson, David Arnold, Andrea~L Bertozzi, Yat~Tin Chow, and Stanley
  Osher.
\newblock Optimal human navigation in steep terrain: a
  hamilton--jacobi--bellman approach.
\newblock {\em Communications in Mathematical Sciences}, 17(1):227--242, 2019.

\bibitem{parkinson2024efficient}
Christian Parkinson and Isabelle Boyle.
\newblock Efficient and scalable path-planning algorithms for curvature
  constrained motion in the hamilton-jacobi formulation.
\newblock {\em Journal of Computational Physics}, 509:113050, 2024.

\bibitem{parkinson2021timeoptimal}
Christian Parkinson and Madeline Ceccia.
\newblock Time-optimal paths for simple cars with moving obstacles in the
  hamilton-jacobi formulation.
\newblock In {\em 2022 American Control Conference (ACC)}, pages 2944--2949.
  IEEE, 2022.

\bibitem{ParkPolage}
Christian Parkinson and Kyle Polage.
\newblock An efficient semi-real-time algorithm for path planning in the
  hamilton-jacobi formulation.
\newblock {\em IEEE Control Systems Letters}, 2023.

\bibitem{peyre2010}
Gabriel Peyr\'{e}, Mickael P\'{e}chaud, Renaud Keriven, and Laurent~D. Cohen.
\newblock Geodesic methods in computer vision and graphics.
\newblock {\em Found. Trends. Comput. Graph. Vis.}, 5(3–4):197–397, mar
  2010.

\bibitem{Qian2006}
Jianliang Qian.
\newblock Approximations for viscosity solutions of hamilton-jacobi equations
  with locally varying time and space grids.
\newblock {\em SIAM Journal on Numerical Analysis}, 43(6):2371--2401, 2006.

\bibitem{reeds}
J.~A. Reeds and L.~A. Shepp.
\newblock {Optimal paths for a car that goes both forwards and backwards.}
\newblock {\em Pacific Journal of Mathematics}, 145(2):367 -- 393, 1990.

\bibitem{Ren2000}
Weiqing Ren and Xiao-Ping Wang.
\newblock An iterative grid redistribution method for singular problems in
  multiple dimensions.
\newblock {\em Journal of Computational Physics}, 159(2):246--273, 2000.

\bibitem{Sethian2000}
J.~A. Sethian and A.~Vladimirsky.
\newblock {Fast methods for the Eikonal and related Hamilton– Jacobi
  equations on unstructured meshes}.
\newblock {\em Proceedings of the National Academy of Sciences},
  97(11):5699--5703, 2000.

\bibitem{sethian2003ordered}
James~A Sethian and Alexander Vladimirsky.
\newblock Ordered upwind methods for static hamilton--jacobi equations: Theory
  and algorithms.
\newblock {\em SIAM Journal on Numerical Analysis}, 41(1):325--363, 2003.

\bibitem{Singh2023}
Ramanjeet Singh, Jing Ren, and Xianke Lin.
\newblock A review of deep reinforcement learning algorithms for mobile robot
  path planning.
\newblock {\em Vehicles}, 5(4):1423--1451, 2023.

\bibitem{Sanchez2021}
José~Ricardo Sánchez-Ibáñez, Carlos~J. Pérez-del Pulgar, and Alfonso
  García-Cerezo.
\newblock Path planning for autonomous mobile robots: A review.
\newblock {\em Sensors}, 21(23), 2021.

\bibitem{Takei2013OptimalTO}
Ryo Takei and Richard Tsai.
\newblock {Optimal Trajectories of Curvature Constrained Motion in the
  Hamilton–Jacobi Formulation}.
\newblock {\em Journal of Scientific Computing}, 54:622--644, 2013.

\bibitem{takei2010}
Ryo Takei, Richard Tsai, Haochong Shen, and Yanina Landa.
\newblock {A practical path-planning algorithm for a simple car: a
  Hamilton-Jacobi approach}.
\newblock In {\em Proceedings of the 2010 American Control Conference}, pages
  6175--6180, 2010.

\bibitem{tsai2003fast}
Yen-Hsi~Richard Tsai, Li-Tien Cheng, Stanley Osher, and Hong-Kai Zhao.
\newblock Fast sweeping algorithms for a class of hamilton--jacobi equations.
\newblock {\em SIAM journal on numerical analysis}, 41(2):673--694, 2003.

\bibitem{Tsitsiklis1995}
J.N. Tsitsiklis.
\newblock {Efficient algorithms for globally optimal trajectories}.
\newblock {\em IEEE Transactions on Automatic Control}, 40(9):1528--1538, 1995.

\bibitem{v1928theorie}
J~v.~Neumann.
\newblock Zur theorie der gesellschaftsspiele.
\newblock {\em Mathematische annalen}, 100(1):295--320, 1928.

\end{thebibliography}
\end{document}